\documentclass[conference]{IEEEtran}
\usepackage{pifont}
\usepackage{bbding}  % Comment this line out                                                         % paper

\IEEEoverridecommandlockouts                              % This command is only
\usepackage{amsmath,amssymb,epsf,epsfig,times}
\usepackage{amsfonts}
\usepackage[all]{xy}
\usepackage{color}
\usepackage{subfigure}
\usepackage{url,cite}
\newtheorem{theorem}{Theorem}[section]
\newtheorem{lemma}{Lemma}[section]

\def\QED{\mbox{\rule[0pt]{1.5ex}{1.5ex}}}
\def\endproof{\hspace*{\fill}~\QED\par\endtrivlist\unskip}
\newcommand{\re}{\mathbb{R}}

\newcommand{\norm}[1]{\left\|#1\right\|}

\newtheorem{assumption}[theorem]{Assumption}

\newtheorem{remark}[theorem]{Remark}
\newtheorem{corollary}[theorem]{Corollary}

\newcommand{\OMIT}[1]{}

\newif\ifpdf
\ifx\pdfoutput\undefined \pdffalse \else \pdfoutput=1 \pdftrue \fi
\ifpdf
\usepackage{graphicx}
\usepackage{epstopdf}
\usepackage{graphicx}
\fi
\title{\LARGE \bf Heterogeneous Distributed Average Tracking}

\author{Salar Rahili, Wei Ren% <-this % stops a space
\thanks{Salar Rahili and Wei Ren are with the Department of Electrical Engineering, University of California, Riverside, CA, USA.
       {Email: srahi001@ucr.edu, ren@ee.ucr.edu}}
}
\begin{document}
%\title{Distributed Coordinated Tracking for Networked Euler-Lagrange systems}
%\markboth{Submitted to IEEE Transactions on Automatical Control}
%         {Submitted to IEEE Transactions on Automatical Control}

\maketitle

\begin{abstract}		
 This paper addresses distributed average tracking for a group of heterogeneous physical agents consisting of single-integrator, double-integrator and Euler-Lagrange dynamics. Here, the goal is that each agent uses local information and local interaction to calculate the average of individual time-varying reference inputs, one per agent.
Two nonsmooth algorithms are proposed to achieve the distributed average tracking goal. In our first proposed algorithm, each agent tracks the average of the reference inputs, where each agent is required to have access to only its own position and the relative positions between itself and its neighbors. To relax the restrictive assumption on admissible reference inputs, we propose the second algorithm.
A filter is introduced for each agent to generate an estimation of the average of the reference inputs. Then, each agent tracks its own generated signal to achieve the average tracking goal in a distributed manner. Finally, numerical example is included for illustration.

	\end{abstract}

	%%%%%%%%%%%%%%%%%%%%%%%%%%%%%%%%%%%%%%%%%%%%%%%%%%%%%%%%%%%%%%%%%%%%%%%%%%%%%%%%
	\section{Introduction}
	
In many applications of multi-agent
systems, agents are required to compute the summation of individual time-varying inputs in a distributed manner. For example, in sensor fusion \cite{olfati2004consensus}, feature-based map merging \cite{aragues2012distributed}, distributed Kalman filtering \cite{bai2011}, and distributed optimization \cite{SalarTac}, computing the average of individual reference inputs is an inseparable part of the algorithms and hence this problem attracted a significant
attention recently.

In this paper,  an average tracking problem for a team of  heterogeneous agents is studied, where each agent uses local information to calculate the average of individual time-varying reference inputs, one per agent.
	Here, the average of individual reference inputs is time-varying and it is not available to any agent; hence  distributed average tracking introduces additional complexities and theoretical challenges compared to the consensus and leader-followers problems.

Researchers have introduced linear distributed algorithms as one of the earlier approaches addressing this problem
\cite{spanos2005dynamic,freeman2006,bai2010,kiaauthority}. 
	In \cite{freeman2006}, a proportional-integral algorithm is proposed to achieve distributed average tracking for slowly-varying
	reference inputs with a bounded tracking
	error.
	In \cite{bai2010}, through the use of the internal model principle, an algorithm is introduced  for a special group of time-varying reference inputs with a common denominator in their Laplace transforms. In \cite{kiaauthority}, a distributed average tracking problem is solved, with steady-state errors, while the privacy of each agent's input is preserved.

	However, in linear algorithms, the reference inputs are required to satisfy restrictive constraints and most of the results only can guarantee to have a bounded error.
Therefore, some results based on nonlinear tracking algorithms have been published recently \cite{Nosrati2012,chen2012distributed}. A class of nonlinear algorithms is introduced in \cite{Nosrati2012}, where it is proved that for reference inputs with bounded deviations the tracking error is bounded. In \cite{chen2012distributed}, a nonsmooth algorithm is proposed for reference inputs with bounded derivatives.
	
	However, all the aforementioned studies addressed the distributed average tracking problem from an estimation perspective, where the agents do not have a certain physical dynamics.
	There are various applications, where the distributed average tracking problem is employed as a control law for physical agents  \cite{cheah2009region}.
	For example, multiple agents moving in a formation with local information and interaction might
need to cooperatively figure out what optimal trajectory the virtual leader or center of the team should follow, where each individual agent specifies its motion using that knowledge. Distributed average tracking can be employed in this problem, where each agent can construct its own reference input using the gradient of its own local cost function \cite{SalarTac}.
	A distributed average tracking algorithm is proposed in \cite{feidoubleintegrator}, for physical agents with double-integrator dynamics, where the reference inputs are allowed to have a bounded accelerations.
A distributed algorithm without using velocity measurements for a group
of physical second-order agents is introduced in \cite{ghapani2015distributed}, where the reference input are assumed to have bounded accelerations' deviations. However, this algorithm is not robust to position and velocity initialization errors. Therefore, it is modified in \cite{sheida-novel} to remove the initialization constraint and communication between agents.

However, in real applications physical agents might have more complicated dynamics rather than single-integrator or double-integrator dynamics. There are only a few studies, that have addressed more complicated dynamics. For example, in  \cite{DBLP:journals/corr/ZhaoDL13}, the problem is studied for physical agents with general linear dynamics, where reference inputs are bounded.
A class of algorithms is proposed in \cite{fei2015EulerDAC}, to achieve distributed average tracking for physical Euler-Lagrange systems, where it is proved that a bounded error is achieved for reference inputs with bounded derivatives.  
In \cite{SalarACC16},  a distributed average tracking algorithm is proposed for physical second-order agents, where there is
a nonlinear term in both agents' and reference inputs' dynamics.

%As it is stated, in many applications the physical agents have to track a time-varying trajectory, where each agent has an incomplete copy of this trajectory. 
In most of the studies in the literature, agents are assumed to be identical. There are only few works assumed nonidentical parameters or nonidentical additive terms in agents' dynamics \cite{fei2015EulerDAC,SalarACC16}.  However, in real applications, we might need to employ different agents (robots) with different abilities to accomplish a task. In these scenarios, agents obey completely different physical dynamics. To the best of our knowledge, the heterogeneous average tracking problem in the literature has been limited to the case that the reference inputs are time-invariant, where the problem is transformed into a  distributed consensus \cite{JieInertia,Single-Double,S-D-E}.
In heterogeneous distributed consensus algorithms, there always exists a term forcing the velocity of each individual agent to zero. This tremendously reduces the complexity of the problem. However, in dynamic average tracking problem, our goal is to track a time-varying trajectory, where a precise control on velocities and accelerations of the agents are required.
It is worthwhile to mention that having a heterogeneous multi-agent system consisting of agents with different dynamics, it is not possible to employ the algorithms proposed for homogeneous dynamics, corresponding to each agent's dynamic, and expect to have a well-behaved system. Therefore, a careful analysis considering the interaction among the agents with different dynamics is needed.

In this paper, a heterogeneous framework consisting of agents with three different dynamics, single-integrator, double-integrator and Euler-Lagrange dynamics, is considered. Two nonsmooth algorithms are proposed to achieve the distributed average tracking goal.
In our first proposed algorithm, each agent is required to have access to only its own position and the relative positions between itself and its neighbors. In some applications, the relative positions can be obtained by using only agents' local sensing capabilities, which might in turn eliminate the communication necessity between agents.
To relax some restrictive assumptions on admissible reference inputs, we propose an estimator-based algorithm, where a filter is introduced for each agent to generate an estimation of the average of the reference inputs. Then, each agent tracks its own generated signal to accomplish the average tracking task. 
In both algorithms, agents described by Euler-Lagrange dynamics, place a restrictive assumption on the admissible reference inputs. The advantage of the second algorithm will be more substantial for a mutli-agent system consisting of agents with only single-integrator and double-integrator dynamics.
In such a framework, using estimator-based algorithm, the heterogeneous dynamic average tracking goal is achieved, where there is no restriction on reference inputs. As a trade-off, the estimator based algorithm necessitates communication between neighbors, where each agent must communicate its own filter's variables with its neighbors.

%{\color{blue}The reminder of this Section is devoted to clarify the advantages of the proposed algorithms in comparison with literature in details.}	
\section{notations and preliminaries} \label{sec:notation}
Throughout the paper, $\mathbb{R}$ denotes the set of all real numbers.
	The transpose of matrix $A$ and vector $x$ are shown as $A^T$ and $x^T$, respectively. 
	Let $\mathbf{1}_n$ and $\mathbf{0}_n$ denote the $n \times 1$ column vector of all ones and all zeros respectively.
	%{\color{red}Let $\lambda_{\max} (.)$ and $\lambda_{\min} (.)$ denote, respectively, the maximum and minimum eigenvalues of a square real matrix with real eigenvalues.} %Let $\mbox{diag}(z_1,\ldots,z_p)$ be the diagonal matrix with diagonal entries $z_1$ to $z_p$.
	Let $\mbox{diag}(a_1,\ldots,a_p)$ be the diagonal matrix with diagonal entries $a_1$ to $a_p$.
	We use %$||\cdot||$ to denote the Euclidean norm,
	$\otimes$ to denote the Kronecker product, and $\mbox{sgn}(\cdot)$ to denote the $\mbox{signum}$ function defined componentwise. 
	For a vector function ${x(t):\re\mapsto\re^m}$, define $\norm{x}_\mathfrak{p}$ as the $\mathfrak{p}$-norm. The cardinality of a set $S$ is
denoted by $|S|$.
%	, 
%	${x(t)\in\mathbb{L}_2}$ if
%	$\int_{0}^{\infty}x(\tau)^T x(\tau)\mbox{d}\tau<\infty$ and
%	${x(t)\in\mathbb{L}_{\infty}}$ if for each element of $x(t)$, denoted as
%	$x_i(t)$, ${\text{sup}_{t \geq 0}|x_i(t)|<\infty}$, $i=1,\ldots,m$.
%%%%%%%%%%%%%%%%%%%%%%%%%%%
	
An \textit{undirected} graph $G \triangleq (V,E)$ is used to characterize the interaction topology among the agents, where ${V \triangleq \{1,\ldots,n\}}$ is the node set and $E \subseteq V \times V$ is the edge set.
	%Let the sizes of $V$ and $E$ are, respectively, represented by $n \triangleq |V|$ and $m \triangleq |E|$.
	An edge $(j,i) \in E$ means that node $i$ can obtain information from node $j$ and vice versa.
	Self edges $(i,i)$ are not considered here.
		The \textit{adjacency matrix} ${\mathbf{A} \triangleq [a_{ij}] \in \mathbb{R}^{n \times n}}$ of the graph $G$ is defined such that the edge weight ${a_{ij}=1}$ if ${(j,i) \in E}$ and ${a_{ij}=0}$ otherwise. For an undirected graph, ${a_{ij}=a_{ji}}$.
	The \textit{Laplacian matrix} ${L \triangleq [l_{ij}] \in \mathbb{R}^{n \times n}}$ associated with $\mathbf{A}$ is defined as ${l_{ii}=\sum_{j \ne i} a_{ij}}$ and ${l_{ij}=-a_{ij}}$, where ${i \ne j}$.
	For an undirected graph, $L$ is symmetric positive semi-definite.
	By arbitrarily assigning an orientation for the edges in $G$, let $D \triangleq [d_{ij}] \in  \mathbb{R}^{n \times |E|}$ be the \textit{incidence matrix} associated with $G$, where $d_{ij} = -1$ if the edge $e_j$ leaves node $i$, $d_{ij} = 1$ if it enters node $i$, and $d_{ij} = 0$ otherwise.
	The \textit{Laplacian matrix} $L$ is then given by $L=DD^T$ \cite{GodsilRoyle01}.
	
	\begin{lemma} \cite{GodsilRoyle01} \label{eigen}
		%\begin{Lemma}[\cite{PetersenPedersen2008}]
		For a connected graph $G$, the \textit{Laplacian matrix} $L$ has a simple zero eigenvalue such that $0=\lambda_1(L)<\lambda_2(L) \leq \ldots \leq \lambda_n(L)$, where $\lambda_i(\cdot)$ denotes the $i$th eigenvalue. Furthermore, for any vector $y \in \mathbb{R}^n$ satisfying ${\mathbf{1}_n^T y=0}$, we have $\lambda_2(L) y^Ty \leq y^T L y \leq \lambda_n(L) y^Ty$.
	\end{lemma}
	
\begin{corollary}\cite{laSalleNonsmooth} \label{barbalate-nonsmooth}
 Consider the system,
 \begin{align}\label{sys1}
 \dot{x}=f(x,t),
\end{align}where $x(t) \in \mathcal{D} \subset\mathbb{R}^n$ and $f: \mathcal{D}\times [0, \infty] \to \mathbb{R}^n$ and $\mathcal{D}$ is an open and connected set containing $x=0$, and suppose is Lebesgue measurable and is essentially locally bounded, uniformly in $t$. Let $V: \mathcal{D}\times [0, \infty] \to \mathbb{R}$
be locally Lipschitz and regular such that
\begin{align}
W_1(x) \leq& V(x,t) \leq W_2(x) \notag\\
&\dot{\tilde{V}} \leq -W(x),
\end{align}
$\forall t \geq 0, \forall x \in \mathcal{D}$, where $W_1$ and $W_2$ are continuous positive definite
functions, and $W$ is a continuous positive semi-definite function
on $ x \in \mathcal{D}$ and $\dot{\tilde{V}}$ is the generalized gradient of function $V$. 
Choose $r>0$ and $c>0$ such that $\mathcal{B}_r \subset \mathcal{D}$ and $c < \underset{\norm{x}=r}\min  W_1(x)$. Then, all Filippov solutions of \eqref{sys1} such that $x(t_0) \in \{x \in \mathcal{B}_r | W_2(x) \leq c\}$ are bounded and satisfy $W(x) \to 0$ as $t \to \infty$.
	\end{corollary}
\section{Problem Statement}\label{Sec:prob-def}
%%%%%%%%%%%%%%%%%%%%%%%%%%%
Consider a heterogeneous multi-agent system consisting of $N$ physical agents, where $\mathcal{I}$ denotes the index set $\{1, \cdots, N\}$. The agents' are described by single-integrator, double-integrator and Euler-Lagrange dynamics. Without loss of generality, we label single-integrator agents as $1,\ldots, M-1$, where their dynamics is described by
\begin{align} \label{single-dynamics}
		\dot{x}_i= u_i,  \qquad  i=1,\ldots, M-1.
	\end{align}
	We also label double-integrator agents as $M,\ldots, N'-1,$ with dynamics described by
\begin{align} \label{double-dynamics}
		\dot{x}_i=& v_i, \quad	\dot{v}_i=u_i,  \qquad  i=M,\ldots, N'-1.
	\end{align}
Agents with Euler-Lagrange dynamics are labeled as $N',\ldots,  N,$ and their dynamic is described by 
\begin{align} \label{euler-dynamics}
		M_i(x_i) \ddot{x}_i+C_i(x_i, \dot{x}_i)\dot{x}_i+g_i(x_i)=u_i  \qquad  i=N',\ldots, N,
	\end{align}   
where $x_i(t) \in \mathbb{R}^\mathfrak{p}$, $v_i(t) \in \mathbb{R}^\mathfrak{p}$ and $u_i(t) \in \mathbb{R}^\mathfrak{p}$ are, respectively, $i$th agent's position, velocity and control input.  $M_i (x_i )$ is the $\mathfrak{p} \times \mathfrak{p}$ symmetric inertia matrix, $C_i (x_i, \dot{x}_i)\dot{x}_i$ is the Coriolis and centrifugal force, and $g_i(x_i)$ is the vector of gravitational force. 
The dynamics of the Lagrange systems satisfy the following properties \cite{Euler-Book}:
\begin{itemize}

\item[(P1)] There exist positive constants $k_{\underline{M}}, k_{\overline{M}}, k_{\overline{C}} , k_{\overline{g}}$ such that $k_{\underline{M}} I_p \le M_i(x_i) \le k_{\overline{M}} I_\mathfrak{p}, ||C_i (x_i,\dot{x}_i )\dot{x}_i || \le k_{\overline{C}} ||\dot{x}_i ||$ and $||g_i (x_i)|| \le k_{\overline{g}}$ .
\item[(P2)] $\dot{M}_i (x_i )-2C_i (x_i,\dot{x}_i )$ is skew symmetric.
\item[(P3)] The Lagrange dynamics can be rewritten as, i.e., $M_i(x_i) \chi+C(x_i, \dot{x}_i)\psi+g_i(x_i)=Y_i (x_i,\dot{x}_i,\chi,\psi) \theta_i$, $\forall \chi,\psi \in \mathbb{R}^\mathfrak{p}$, where $Y_i \in \mathbb{R}^{\mathfrak{p} \times \mathfrak{p}_ \theta}$ is the regression matrix and $\theta_i \in \mathbb{R}^{\mathfrak{p}_\theta}$ is the unknown but constant parameter vector.
\end{itemize}

In our framework the agents' interaction topology is described by an undirected graph $G$. 
	\begin{assumption} \label{conn-graph}
		Graph $G$ is connected.
	\end{assumption}
	
	Suppose that each agent has a time-varying reference input $r_i(t) \in \mathbb{R}^\mathfrak{p}$, $i \in \mathcal{I}$, satisfying
	\begin{align} \label{ref-dynamic}
	\dot{r}_i(t)=& v_i^r(t),   \notag \\
	\dot{v}_i^r(t)=& a_i^r(t),
	\end{align}
	where $v_i^r(t) \in \mathbb{R}^\mathfrak{p}$ and $a_i^r(t) \in \mathbb{R}^\mathfrak{p}$ are, respectively, the reference velocity and the reference acceleration for agent $i$ at time $t$.	

\begin{assumption}\label{bound-r-vr}
The reference input $r_i(t), \forall i \in \mathcal{I}$ and its velocity $v^r_i(t)$ are bounded. It is assumed that $\norm{r_i(t)}<\bar{r},$ and $\norm{v^r_i(t)}<\bar{v^r}, \forall i \in \mathcal{I}$, where $\bar{r}$ and $\bar{v^r}$ are positive constants.
\end{assumption}

Here the goal is to design $u_i(t)$ for agent $i \in \mathcal{I}$, to track the average of the reference inputs, i.e.,
	\begin{align}\label{goal}
	\lim \limits_{t \to \infty} ||x_i(t)-\frac{1}{N} \sum_{j=1}^N r_j(t)||=&0,
	\end{align}
where each agent has only local interaction with its neighbors.
 	
%In this paper, two algorithms are proposed to achieve the distributed average tracking goal \eqref{goal} for a heterogeneous framework. In our first proposed algorithm, each agent has access to only its own position and the relative positions between itself and its neighbors. To relax some restrictive assumption on admissible reference signals, we propose the second algorithm, where a filter is introduced for each agent. However, in our second algorithm, each agent must communicate its own filter's variables with its neighbors.

	\subsection{Distributed Average Tracking for Heteregeous Physical Agents Using Neighbors' Positions}\label{subsec:NoFilter}
	In this subsection, we study the distributed average tracking problem for heterogeneous multi-agent system consisting of three different dynamics, single-integrator, double-integrator and Euler-Lagrange dynamics. Here, we propose an algorithm to achieve goal \eqref{goal}, where each agent is required to have access to only its own position and the relative positions between itself and its neighbors. Note that in some applications, these pieces of information can be obtained by sensing; hence the communication necessity might be eliminated. For notational simplicity, we will remove the index $t$ from variables in the reminder of the paper.
	
Three controllers are proposed, where each agent according to its dynamic will employ the proper control $u_i$.  Consider the control input
\begin{align}\label{u-single-i}
		u_i= -  \beta_{i}& \text{sgn}\big[ \sum\limits_{j=1}^{N}  a_{ij}  (x_i-x_j)\big]  \notag\\&- (x_i-r_i)+v^r_i,
		  \qquad  i=1,\ldots, M-1,
		\end{align}  
		for agents with single-integrator dynamics and
		\begin{align} \label{u-double-i}
u_i=& -  \beta_{i} \text{sgn}\big[ \sum\limits_{j=1}^{N}  a_{ij}  (x_i-x_j)\big] - \sum\limits_{j=1}^{N}  a_{ij}  (x_i-x_j) \\&- (x_i-r_i)-2(v_i -v^r_i)+a^r_i, \qquad  i=M,\ldots, N'-1 \notag
	\end{align}  
	for agents with double-integrator dynamics and
	\begin{align} \label{u-euler-nofliter}
u_i=& Y_i (x_i,\dot{x}_i,\upsilon_i,\nu_i) \hat{\theta}_i-\alpha s_i\notag\\ \qquad \ \ &- \sum\limits_{j=1}^{N}  a_{ij}  (x_i-x_j) ,  \qquad  i=N',\ldots, N\ \notag\\		
		\nu_i=&-\beta_{i}  \text{sgn}\big[ \sum\limits_{j=1}^{N}  a_{ij}  (x_i-x_j)\big]-(x_i-r_i)+v^r_i, \notag\\
		s_i=& \dot{x}_i-\nu_i,\\
	 \dot{\hat{\theta}}_i =&-Y_i (x_i,\dot{x}_i,\upsilon_i,\nu_i)^T s_i,\notag
	\end{align}  		
for agents with Euler-Lagrange dynamics, where  $\alpha$ and $\beta_{i}$ are positive constant gains to be designed, and ${\hat{\theta}}_i$ is  the estimate of the unknown but constant parameters $\theta_i$.
Using the definition of the generalized gradient \cite{clarke}, the generalized time-derivative of $s_i$ and $\nu_i$ are defined, respectively, as $\vartheta_i$ and $\upsilon_i$, where $\zeta_i \in \vartheta_i$, and $\mu_i \in \upsilon_i$. Let $\xi_i$ denotes the minimum norm element of $\upsilon_i$.

\begin{theorem} \label{Theorem-NoFilter}
Under the control law given by  \eqref{u-single-i}-\eqref{u-euler-nofliter} for system defined in \eqref{single-dynamics}-\eqref{euler-dynamics}, distributed average tracking goal \eqref{goal} is achieved asymptotically, provided that Assumptions \ref{conn-graph} and \ref{bound-r-vr} hold and the control gain $\beta_{i}$ is chosen such that $\min_{i\in \mathcal{I}} \beta_i> \bar{r}+\bar{v^r}$ and $\alpha>0$.
		\end{theorem}
	\begin{proof}
Rewrite the Laplacian matrix as $L=[L^T_s\  L^T_d \ L^T_e]^T$, where $L_s \in \mathbb{R}^{(M-1)\times N}, L_d \in \mathbb{R}^{(N'-M)\times N}$ and $L_e \in \mathbb{R}^{(N-N'+1)\times N}$ and subscripts $_s$, $_d$ and $_e$, respectively, are used for single-integrator, double-integrator and Euler-Lagrange dynamics, i.e, $L_s$ describes the interaction among single-integrator agents and other agents.
Let $x$ denotes the column stack vectors of all $x_i$'s $i=1,\ldots,N$, and it can be rewritten as $x=[x^T_s \ x^T_d\ x^T_e]^T,$ where $x_s, x_d$ and $x_e$ are, respectively, the column stack vectors of the positions for single-integrator, double-integrator and Euler-Lagrange dynamics. 

System \eqref{single-dynamics} with control input \eqref{u-single-i} can be rewritten in vector form as
\begin{align}\label{closed-single}
		\dot{x}_s=& -\beta_s \text{sgn}[(L_s\otimes I_\mathfrak{p}) x]-(x_s-r_s)+v^r_s,
				\end{align}  
where $r_s=[r_1^T,...,r_{M-1}^T]^T$, and $v^r_s=[{v^r_1}^T,...,{v^r_{M-1}}^T]^T$, denote, respectively, the aggregated reference inputs and reference velocities of the single-integrator dynamic \eqref{single-dynamics} and $\beta_s=\text{diag}(\beta_1,\cdots,\beta_{M-1})$. System \eqref{double-dynamics} with control input \eqref{u-double-i} can be rewritten in vector form as
\begin{align} \label{closed-double}
		\dot{x}_d=&v_d                                               \notag              \\
		\dot{v}_d=& -  \beta_d \text{sgn}[(L_d\otimes I_\mathfrak{p})x]-(L_d\otimes I_\mathfrak{p})x-(x_d-r_d)\\
		&-2(v_d-v^r_d)+a^r_d, \notag
	\end{align}
where $v_d=[{v_M}^T,...,{v_{N'-1}}^T]^T$, $r_d=[r_M^T,...,r_{N'-1}^T]^T$, $v^r_d=[{v^r_M}^T,...,{v^r_{N'-1}}^T]^T$, and $a^r_d=[{a^r_M}^T,...,{a^r_{N'-1}}^T]^T$, denote, respectively, the aggregated velocities, reference inputs, reference velocities and reference accelerations of the double-integrator system \eqref{double-dynamics} and $\beta_d=\text{diag}(\beta_{M},\cdots,\beta_{N'-1})$.

It follows form (P3) that $M(x_e) \zeta+C(x_e, \dot{x}_e) s+Y (x_e,\dot{x}_e,\upsilon,\nu) \theta=u_e$, where $M(x_e) \triangleq \text{diag}\{M_{N'}(x_{N'}),\cdots, M_{N}(x_{N})\}$,  $C(x_e, \dot{x}_e) \triangleq \text{diag}\{C_{N'}(x_{N'}, \dot{x}_{N'}),\cdots, C_{N}(x_{N}, \dot{x}_{N})\}$, and $u_e=[u_{N'}^T,...,u_{N}^T]^T$. Now, by replacing the control input \eqref{u-euler-nofliter}, we have
	\begin{align} \label{closed-loop-Euler}
	M(x_e) \zeta&+C(x_e, \dot{x}_e) s+Y (x_e,\dot{x}_e,\upsilon,\nu) \theta\\ &=Y (x_e,\dot{x}_e,\upsilon,\nu)\hat{\theta}-\alpha s-(L_e\otimes I_\mathfrak{p})x,\notag
	\end{align}	
where $\zeta$, $s$, $\nu$,  $\theta$ and $\hat{\theta}$ are, respectively, the column stack vectors of all $\zeta_i$'s, $s_i$'s, $\nu_i$'s, $\theta_i$'s and $\hat{\theta}_i$'s, $i=N',\ldots,N$. Let $\beta_e=\text{diag}(\beta_{N'},\cdots,\beta_{N})$.
Let $r=[r^T_s \ r^T_d \ r^T_e]^T,$ and $v^r=[{v^r}^T_s {v^r}^T_d \ {v^r}^T_e]^T$ denote, respectively, the aggregated reference inputs, and reference velocities for all agents. 

Define the Lyapunov function $V_t$ as
\begin{align}
V_t=\frac{1}{2}x^T(L\otimes I_\mathfrak{p})& x +\frac{1}{2} (v_d - \Phi)^T (v_d - \Phi)\\&+\frac{1}{2} s M s+\frac{1}{2} \tilde{\theta}^T \tilde{\theta},\notag
\end{align}
where  $\Phi(x)=-\beta_d \text{sgn}[(L_d\otimes I_\mathfrak{p})x]-(x_d-r_d)+v^r_d$, and $\tilde{\theta}=\hat{\theta}-\theta$. It is easy to see that we have $V_1=\frac{1}{2} x^T L x=\frac{1}{2} e^T e,$ where $e=D^Tx$ and $D$ is defined in Section \ref{sec:notation}. Hence $V_1$ is a positive definite function corresponding to $e$. The candidate Lyapunov function $V_t$ satisfies the following inequalities:
		\begin{align}\label{lyap-inequaltiy}
W_1(y) \leq V(y,t) \leq W_2(y),
\end{align}
	where $y =\begin{bmatrix} x\\  v_d-\Phi \\ s\\ \tilde{\theta} \end{bmatrix}$ and $W_1$ and $W_2$ are positive-definite continuous functions defined as $W_1=\Lambda_1 \norm{y}^2$ and $W_2=\Lambda_2 \norm{y}^2$, where $\Lambda_1$ and $\Lambda_2$ are positive constants.

Define the generalized gradient of $V_t$ and $\Phi$ by $ \dot{\tilde{V}}_t$ and $\dot{\tilde{\Phi}}$, respectively. Every element of $\eta \in \dot{\tilde{V}}$ satisfies 
	\begin{align*}\small
		\eta \leq& -x^T (L\otimes I_\mathfrak{p}) \bar{\beta} \begin{bmatrix}
         \text{sgn}[(L_s\otimes I_\mathfrak{p})x]\\
            \text{sgn}[(L_d\otimes I_\mathfrak{p})x] \\  \text{sgn}[(L_e\otimes I_\mathfrak{p})x]\end{bmatrix}\\
                 &+ x^T (L\otimes I_\mathfrak{p}) \bigg( \begin{bmatrix}
        0\\ v_d-\Phi \\s  \end{bmatrix}+ \begin{bmatrix}
        -x_s+r_s+v^r_s\\
         -x_d+r_d+v^r_d \\ -x_e+r_e+v^r_e \end{bmatrix}\bigg)
            	\end{align*}
               		\begin{align*}\small
         &+(v_d - \Phi)^T \times \bigg(-  \beta_d \text{sgn}[(L_d\otimes I_\mathfrak{p})x] -(L_d\otimes I_\mathfrak{p})x\\&-(x_d-r_d)-2(v_d-v^r_d)+a^r_d -\varrho\bigg)+\frac{1}{2} s \dot{M} s + \tilde{\theta}^T \dot{\hat{\theta}}\\
        &+s^T \big(-C(x_e, \dot{x}_e) s-Y (x_e,\dot{x}_e,\upsilon,\nu) \tilde{\theta}-\alpha s\big)\\
        &-s^T(L_e\otimes I_\mathfrak{p})x, \normalsize
   	\end{align*}
	where $\bar{\beta}=\text{diag}(\beta_s, \beta_d,\beta_e)$, and $\varrho \in \dot{\tilde{\Phi}}$ and we used the fact that we can rewrite equations \eqref{closed-double} and \eqref{u-euler-nofliter}, respectively, as \begin{align*}
	{x}_d=\Phi+v_d-\Phi
	\end{align*}and
		\begin{align}\label{s-x}
	\dot{x}_e=-\beta_e \text{sgn}[(L_e\otimes I_\mathfrak{p})x]-(x_e-r_e)+v^r_e+s.
	\end{align}	
Employing \eqref{u-euler-nofliter}, we have 
		\begin{align*}
		\eta \leq& -(\min_{i\in \mathcal{I}}{\beta_i}) \norm{(L\otimes I_\mathfrak{p})x}_1- x^T(L\otimes I_\mathfrak{p})x\\
        &+x^T(L\otimes I_\mathfrak{p})(r+v^r)+x^T (L_d\otimes I_\mathfrak{p})^T (v_d-\Phi)\\
        & +x^T (L_e\otimes I_\mathfrak{p})^T s+(v_d - \Phi)^T [\Phi-v_d-(L_d\otimes I_\mathfrak{p})x\\
        &-(v_d-v^r_d)+a^r_d -\varrho]X+\frac{1}{2} s \dot{M} s-\alpha s^Ts\\
        &+s^T[-C(x_e, \dot{x}_e) s-Y (x_e,\dot{x}_e,\upsilon,\nu) \tilde{\theta}]-\alpha s^Ts\\
        &-s^T(L_e\otimes I_\mathfrak{p})x - \tilde{\theta}^T Y (x_e,\dot{x}_e,\upsilon,\nu)^Ts\\
        =&  -(\min_{i\in \mathcal{I}}{\beta_i}) \norm{(L\otimes I_\mathfrak{p})x}_1- x^T(L\otimes I_\mathfrak{p})x\\&+x^T(L\otimes I_\mathfrak{p})(r+v^r)+(v_d - \Phi)^T (\Phi-v_d)\\&+(v_d - \Phi)^T [(v^r_d-v_d)+a^r_d -\varrho-\chi+\chi]-\alpha s^Ts,
	\end{align*}
where $\chi$ is the minimum norm element of $\dot{\tilde{\Phi}}$ and we have used property (P2) to obtain the last equality.

Under assumption \ref{bound-r-vr} and by selecting $\min_{i\in \mathcal{I}}{\beta_i}> \bar{r}+\bar{v^r}$, we know that $-(\min_{i\in \mathcal{I}}{\beta_i}) \norm{(L\otimes I_\mathfrak{p})x}_1+x^T(L\otimes I_\mathfrak{p})(r+v^r)<0$. Now, using the fact that $\chi=-(v_d-v^r_d)+a^r_d$, it follows $\eta \leq- x^T(L\otimes I_\mathfrak{p})x+(v_d - \Phi)^T (\Phi-v_d)+(v_d - \Phi)^T(\chi-\varrho)-\alpha s^Ts$.
Using an argument similar to \cite{Kyriakopoulos2008}, $\chi-\varrho$ is zero wherever $\nu(x,v,t)$ is differentibale. Also at points of non-differentiability, we will have $\chi-\varrho=0$ \cite{Kyriakopoulos2008}.
	Hence, $\eta \leq - x^T(L\otimes I_\mathfrak{p})x-(v_d - \Phi)^T (v_d - \Phi)-\alpha s^Ts$.
Now, we can see that $\dot{\tilde{V}}_t \leq -W(y)$, where $W$ is a  positive semi-definite defined on the domain $\mathcal{D}= \mathbb{R}^{N(3\mathfrak{p}+{\mathfrak{p}_\theta})}$. As a result $V_t \in \mathcal{L}_\infty$, and  $\tilde{\theta}, s, e, (v_d - \Phi) \in \mathcal{L}_\infty$. 

By calling $z=x_e-r_e$, we can rewrite \eqref{s-x} as $\dot{z}=-z-\beta_e \text{sgn}[(L\otimes I_\mathfrak{p})x]+s$, where we know that $(L\otimes I_\mathfrak{p})x$ and $s$ are bounded. Hence it is easy to see that $z$ will remain bounded. Also $\dot{z}$ will be bounded because $z, (L\otimes I_\mathfrak{p})x$ and $s$ are bounded. Now, using \eqref{closed-loop-Euler} and under assumptions (P1) and (P3), it is easy to see that  $\zeta$ is bounded. 
	
Knowing the fact that $s$ is continuous and bounded, we can use the mean value theorem for nonsmooth functions \cite{meanvalue}, where we have \begin{align} 
	\frac{s(t_1)-s(t_0)}{t_1-t_0} \in S,\  \forall t_0, t_1
\end{align}
and $S$ denotes the set $\partial s(t) \cup- \partial (-s)(t)$ for $t \in (t_0, t_1)$. Because $\zeta$ is bounded for every $\zeta \in \vartheta$, we know that there exists a $\kappa$ such that  $\partial s(t)\leq \kappa, \forall t$.  Hence, the members of the set $S$ are all bounded and we have $s(t_1)-s(t_0) \leq \kappa (t_1-t_0),  \forall t_0, t_1$, which shows $s$ is lipschitz and therefore it is uniformly continuous.

Now, choose $\rho>0$ such that  $\mathcal{B}_\rho \subset \mathcal{D}$ denotes a closed ball. Define $\mathcal{M}\subset \mathcal{D}$ as $\mathcal{M} \triangleq \{\varpi \subset \mathcal{M} | W_2(\varpi) \leq \underset{\norm{\varpi}=\rho}\min  W_1(\varpi)=\lambda_1\rho^2\}$.  Then, all conditions in Corollary \ref{barbalate-nonsmooth}, LaSalle-Yoshizawa for nonsmooth systems, are provided and we have $W(y) \to 0$ as $t \to \infty$, $\forall y(0) \in \mathcal{M}$. Because $\rho$ can be selected arbitrarily large to include all
initial conditions, the region of attraction is $\mathcal{M}= \mathbb{R}^{N(3\mathfrak{p}+{\mathfrak{p}_\theta})}$ .

Now, having $W(y) \to 0$, it follows that $s \to 0$, $v_d - \Phi \to 0$ and $(L\otimes I_\mathfrak{p})x \to 0$.
Since $v_d - \Phi \to 0$, we will have
\begin{align}\label{epsilon-d}
\dot{x}_d=-\beta_d \text{sgn}[(L_d\otimes I_\mathfrak{p}) x]-(x_d-r_d)+v^r_d+\epsilon_d,
\end{align}
where $\epsilon_d \to 0$ as $t \to \infty$.
Also using \eqref{s-x}, and because $s \to 0$, we have 
\begin{align} \label{epsilon-e}
	\dot{x}_e=-\beta_e \text{sgn}[(L_e\otimes I_\mathfrak{p})x]-(x_e-r_e)+v^r_e+\epsilon_e,
\end{align}
where $\epsilon_e \to 0$ as $t \to \infty$.

Hence, it turns out that using  \eqref{closed-single}, \eqref{epsilon-d} and \eqref{epsilon-e}, we have
\begin{align}
\dot{x}_s=-\beta_s &\text{sgn}[(L_s\otimes I_\mathfrak{p})x]-(x_s-r_s)+v^r_s,\\
\dot{x}_d=-\beta_d &\text{sgn}[(L_d\otimes I_\mathfrak{p})x]-(x_d-r_d)+v^r_d+\epsilon_d,\\
\dot{x}_e=-\beta_e &\text{sgn}[(L_e\otimes I_\mathfrak{p})x]-(x_e-r_e)+v^r_e+\epsilon_e,
\end{align}
where we can rewrite it as
\begin{align}\label{sys-close-all}
\dot{x}= -\bar{\beta}\text{sgn}[(L\otimes I_\mathfrak{p})x]-(x-r)+v^r+\epsilon,
\end{align}
where $\epsilon=\begin{bmatrix}
           0 \\
          \epsilon_d\\
          \epsilon_e
         \end{bmatrix}$. Define the Lyapunov candidate function $V_1=x^T(L\otimes I_\mathfrak{p})x$, where its time-derivative along the system \eqref{sys-close-all} is
\begin{align}\label{Lyap-sys-close-all}
\dot{V_1}=& -\bar{\beta} x^T(L\otimes I_\mathfrak{p})\text{sgn}[(L\otimes I_\mathfrak{p})x-x^T(L\otimes I_\mathfrak{p})x \notag\\&+x^T(L\otimes I_\mathfrak{p})(r+v^r)+x^T(L\otimes I_\mathfrak{p})\epsilon \notag\\
         &\leq -\bar{\beta}\norm{(L\otimes I_\mathfrak{p})x}_1-x^T(L\otimes I_\mathfrak{p})x\notag\\&+x^T(L\otimes I_\mathfrak{p})(r+v^r)+x^T(L\otimes I_\mathfrak{p})\epsilon
\end{align}
Now, by by selecting $\min_{i\in \mathcal{I}}{\beta_i}> \bar{r}+\bar{v^r}$ and knowing that $\epsilon \to 0$ as $t \to \infty$, we can employ Lemma 2.19 in \cite{Qu}. As a result we can show that the agents' positions reach consensus, i.e, $x_i=x_j,$ as $t\to \infty$. 
Define the variable $S_1=(\textbf{1}_N^T\otimes I_\mathfrak{p})(x-r)=\sum_{i=1}^Nx_i-\sum_{i=1}^Nr_i$, then we can rewrite \eqref{sys-close-all} as
\begin{align}\label{S1}
\dot{S}_1= -(\textbf{1}_N^T\otimes I_\mathfrak{p})\bar{\beta}\text{sgn}(Lx)-S_1+\epsilon.
\end{align}
Then we can use input-to-state stability to analyze the system \eqref{S1}
by treating the term $(\textbf{1}_N^T\otimes I_\mathfrak{p})\bar{\beta}\text{sgn}(Lx)$ as the input and $S_1$ as the state.
The system \eqref{S1} with zero input
is exponentially stable and hence input-to-state stable. Since $Lx \to 0$ as $t \to \infty$ for each agent, it follows that $S_1 \to 0$ as $t \to \infty$. This implies that $\sum_{i=1}^Nx_i \to \sum_{i=1}^Nr_i$, where combining it with the consensus result, we will have
\begin{align}
x_i \to \frac{1}{N} \sum_{j=1}^N r_j,\ \forall i \in \mathcal{I}.
\end{align}         
\end{proof}	
\begin{remark}
Note that the controllers in \eqref{u-single-i}-\eqref{u-euler-nofliter} are proposed precisely for our heterogeneous framework and they are not just a simple combination of the controllers in the literature. The interaction among agents with different dynamics is one of the challenge that we have faced. The only common state among our agents is position; hence we cannot use the well-known algorithms for double-integrator or Euler-Lagrange dynamics, which they require velocity measurement or communication. It is worthwhile to mention that algorithm \eqref{u-double-i} is proposed based on the intuition behind Backstepping approach.
\end{remark}
\subsection{Estiamtor Based Distributed Average Tracking for Heteregeous Physical Agents} \label{subsec:WithFilter}
	In this subsection, we propose an estimator based algorithm to address the distributed average tracking problem \eqref{goal} for heterogeneous multi-agent systems \eqref{single-dynamics}-\eqref{euler-dynamics}.  Here, a filter is used to generate the average of the inputs in a distributed manner, where each agent tracks its own generated signal. In some frameworks, the estimator based algorithm is able to relax the restrictive assumptions mentioned in Subsection \ref{subsec:NoFilter}. As a trade-off the estimator based algorithm necessitates communication between neighbors. 

	First, a filter is introduced for each agent to estimate the average of the reference inputs and reference velocities.
	Then the control input $u_i$, $i=1,\ldots,N$, is designed for each agent such that $x_i$ tracks $p_i$, where $p_i \in \mathbb{R}^\mathfrak{p}$ is the filter's output. The filter, adapted from \cite{SalarACC16}, is proposed as following
		\begin{align}  \label{filter-double-i}
		\dot{p}_i=&q_i                                                          \notag    \\
		\dot{q}_i=& -  \beta_{i} \text{sgn}\big[ \sum\limits_{j=1}^{N}  a_{ij} \{ (p_i+q_i)-(p_j+q_j)\}\big] \\
		&- \kappa (p_i-r_i)- \kappa (q_i -v^r_i)+a^r_i, \qquad  i=1,\ldots, N \notag
	\end{align}  
where $\beta_i=\eta_i\norm{r_i}_1+\eta_i \norm{v^r_i}_1+ \norm{a^r_i}_1+\gamma$ is a state based gain and $\eta_i, \gamma$ and $\kappa$ are positive constants to be designed. The controllers are given by
\begin{align}\label{ucontrol-single-i}
		u_i= -  \eta_{i}& \text{sgn} (x_i-p_i)+q_i   \qquad  i=1,\ldots, M-1, 
		\end{align}  
		for agents with single-integrator dynamics and
		\begin{align} \label{ucontrol-double-i}
u_i=&-  \eta_{i}\text{sgn} [(x_i-p_i)+(v_i-q_i)]-\eta_i (x_i-p_i)\notag\\
&-\eta_i(v_i-q_i)+\dot{q}_i, \qquad  i=M,\ldots, N'-1,
	\end{align}  
	for agents with double-integrator dynamics and
	\begin{align} \label{ucontrol-euler-nofliter}
u_i=& Y_i(x_i,\dot{x}_i,p_i,q_i,\dot{q}_i) \hat{\theta}_i-\alpha s_i \qquad  i=N',\ldots, N\ \notag\\		
		s_i=&\mu (x_i-p_i)+(\dot{x}_i-q_i),\\
	 \dot{\hat{\theta}}_i =&-Y_i(x_i,\dot{x}_i,p_i,q_i,\dot{q}_i)^T s_i,\notag
	\end{align}  	 		
for agents with Euler-Lagrange dynamics, where $\alpha$ and $\mu$ are positive constants.
%%%%%%%%%%%%%%%%%%%%%%%%%%%	
\begin{theorem} \label{thm:DAT-filter}
Under the control law given by  \eqref{filter-double-i}-\eqref{ucontrol-euler-nofliter} for system defined in \eqref{single-dynamics}-\eqref{euler-dynamics}, the distributed average tracking goal \eqref{goal} is achieved asymptotically, provided that Assumptions \ref{conn-graph} and \ref{bound-r-vr} hold and the control gains are chosen such that $\eta_i >  \kappa>1$ and $\gamma, \alpha$ and $\mu$ are positive constants. 
\end{theorem}

\emph{Proof}:
\textbf{Filter}:
Here, it is proved that, $\forall i=1,\cdots,N$, we have
\begin{align}\label{estiamtor-goal}
\lim\limits_{t \to \infty} p_i =& \frac{1}{N} \sum\limits_{j=1}^{N} r_j\notag\\
\lim\limits_{t \to \infty} q_i =& \frac{1}{N} \sum\limits_{j=1}^{N} v^r_j.
\end{align}

Let $p=[p_1^T,\ldots,p_N^T]^T$, and $q=[q_1^T,\ldots,q_N^T]^T$, denote the aggregated states of the filters. Let $M\triangleq I_N- \frac{1}{N}\textbf{1}^T_N \textbf{1}_N$. Note that $M$ has one simple zero eigenvalue with $\mathbf{1}_N$ as its right eigenvector and has $1$ as its other eigenvalue with the multiplicity $N-1$.
Define the consensus error vectors $\tilde{p}=(M \otimes I_\mathfrak{p}) p$ and $\tilde{q}=(M \otimes I_\mathfrak{p}) q$.
Then it is easy to see that $\tilde{p}=0$ (respectively, $\tilde{q}=0$) if and only if $p_i=p_j,\ \forall i,j \in {\mathcal I}$ ($q_i=q_j, \forall i,j \in {\mathcal I})$.

Now, the estimator dynamics \eqref{filter-double-i} can be rewritten in vector form as
\begin{align*}\small
	\dot{\tilde{p}}=& \tilde{q}, \notag \\
	\dot{\tilde{q}}=& -\alpha (M  \beta \otimes I_\mathfrak{p}) \mbox{sgn} [( L \otimes I_\mathfrak{p})(\tilde{p}+\tilde{q})]-\kappa \tilde{p} \notag \\
	&  +\kappa  (M  \otimes I_\mathfrak{p}) r - \kappa \tilde{q} +\kappa (M  \otimes I_\mathfrak{p}) v^r+(M \otimes I_\mathfrak{p}) a^r, \notag	
\end{align*}
\normalsize
where  ${r(t)=[r^T_1,\ldots,r^T_N ]^T}$,  ${v^r(t)=[{v_1^r}^T,\ldots,{v_N^r}^T ]^T}$ and ${a^r(t)=[{a^r}^T_1,\ldots,{a^r}^T_N ]^T}$, are respectively, the aggregated reference inputs, reference velocities and reference accelerations, and $\beta=\mbox{diag}(\beta_1,\ldots,\beta_N).$ 

Consider the Lyapunov function candidate
$V_1= \frac{1}{2} 
	\begin{bmatrix}
		\tilde{p}^T && \tilde{q}^T
	\end{bmatrix}
	( L \otimes
	\begin{bmatrix}
		2 \kappa  && 1 \\
		1  && 1
	\end{bmatrix}
	\otimes I_\mathfrak{p})
	\begin{bmatrix}
		\tilde{p} \\
		\tilde{q}
	\end{bmatrix}$.
Since $(\mathbf{1}_N \otimes I_\mathfrak{p}) ^T \tilde{p} =\mathbf{0}_{N\mathfrak{p}}$ and $(\mathbf{1}_N \otimes I_\mathfrak{p}) ^T \tilde{q} =\mathbf{0}_{N\mathfrak{p}}$, by using Lemma \ref{eigen}, we have
$V_1 \geq \frac{\lambda_2(L)}{2} \begin{bmatrix}
		\tilde{p}^T && \tilde{q}^T
	\end{bmatrix} (
	\begin{bmatrix}
		2 \kappa && 1 \\
		1 && 1
	\end{bmatrix} \otimes I_{N\mathfrak{p}})
	\begin{bmatrix}
		\tilde{p} \\
		\tilde{q}
	\end{bmatrix}$,
where $\lambda_2 (L)$ is defined in Lemma \ref{eigen}. Now, using the fact that $\begin{bmatrix}
		2 \kappa && 1 \\
		1 && 1
\end{bmatrix}>0$,
for $\kappa >\frac{1}{2}$, it is easy to see that $V_1$ is positive definite.
The derivative of $V_1$ is given as

\begin{align*}
	\small
	\dot{V}_1 =&  2 \kappa \tilde{p}^T (L \otimes I_\mathfrak{p}) \tilde{q}+\tilde{q}^T (L \otimes I_\mathfrak{p}) \tilde{q}-\kappa \tilde{p}^T (L \otimes I_\mathfrak{p}) \tilde{p} \notag \\
	&+\tilde{p}^T (L \otimes I_\mathfrak{p})\big(\kappa r +\kappa v^r+a^r \big)-\kappa \tilde{p}^T (L \otimes I_\mathfrak{p}) \tilde{q} \notag \\
	&- \tilde{p}^T (L  \beta \otimes I_\mathfrak{p}) \mbox{sgn} [( L \otimes I_\mathfrak{p})(\tilde{p}+\tilde{q})] \notag \\
	&+\tilde{q}^T (L \otimes I_\mathfrak{p})\big(\kappa r +\kappa v^r+a^r \big)-\kappa \tilde{q}^T (L \otimes I_\mathfrak{p}) \tilde{p} \notag \\
	& -\kappa \tilde{q}^T (L \otimes I_\mathfrak{p}) \tilde{q}- \tilde{q}^T (L  \beta \otimes I_\mathfrak{p}) \mbox{sgn} [( L \otimes I_\mathfrak{p})(\tilde{p}+\tilde{q})] \notag \\
	=&-\kappa \tilde{p}^T (L \otimes I_\mathfrak{p}) \tilde{p} - (\kappa -1 ) \tilde{q}^T (L \otimes I_\mathfrak{p}) \tilde{q} \\
	&+(\tilde{p}+\tilde{q})^T (L \otimes I_\mathfrak{p})\big(\kappa r +\kappa v^r+a^r \big) \\
	&-( \tilde{p}+\tilde{q})^T (L  \beta \otimes I_\mathfrak{p}) \mbox{sgn} [( L \otimes I_\mathfrak{p})(\tilde{p}+\tilde{q})], \notag 
\end{align*}\normalsize
where  we have used $LM=L$. Now using the triangular inequality, we have
\begin{align*}
\small
	\dot{V}_1 \leq&-\kappa \tilde{p}^T (L \otimes I_\mathfrak{p}) \tilde{p} - (\kappa -1 ) \tilde{q}^T (L \otimes I_\mathfrak{p}) \tilde{q} \\
	&+ \sum\limits_{i=1}^{N} \Big \| \sum\limits_{j=1}^{N} a_{ij} \Big \{ (\tilde{p}_i+\tilde{q}_i)-(\tilde{p}_j+\tilde{q}_j) \Big \} \Big \|_1 \times \\
	& (\kappa \|r_i\|_1+\kappa\|v^r_i\|_1+\norm{a^r_i}_1) \\
	&-  \sum\limits_{i=1}^{N} \beta_i \Big \|   \sum\limits_{j=1}^{N} a_{ij} \Big \{ (\tilde{p}_i+\tilde{q}_i)-(\tilde{p}_j+\tilde{q}_j) \Big \} \Big \|_1 
	\end{align*}
\normalsize
\begin{align*}
\small
	= & -\kappa \tilde{p}^T (L \otimes I_\mathfrak{p}) \tilde{p} - (\kappa -1 ) \tilde{q}^T (L \otimes I_\mathfrak{p}) \tilde{q} \\
	&+\sum\limits_{i=1}^{N} \Big ( (\kappa-\eta_i) \|r_i\|_1+(\kappa-\eta_i)\|v^r_i\|_1 - \gamma\Big )  \times \\
	& \Big \| \sum\limits_{j=1}^{N} a_{ij} \Big \{ (\tilde{p}_i+\tilde{q}_i)-(\tilde{p}_j+\tilde{q}_j) \Big \}  \Big \|_1 ,
\end{align*}
\normalsize
where $\tilde{p}_i$ and $\tilde{q}_i$ are, respectively, the $i$th components of $\tilde{p}$ and $\tilde{q}$ and we have used the definition of $\beta_i$ to obtain the last equality.
Since $\eta_i >  \kappa$, we will have
\begin{align*}
\dot{V}_1 \leq & -\kappa \tilde{p}^T (L \otimes I_\mathfrak{p}) \tilde{p} - (\kappa -1 ) \tilde{q}^T (L \otimes I_\mathfrak{p}) \tilde{q} \\ 
\leq & -\kappa \lambda_2(L) \tilde{p}^T  \tilde{p}-(\kappa -1)\lambda_2(L) \tilde{q}^T  \tilde{q} < 0,
\end{align*}
where we have used Lemma \ref{eigen}, and  $\kappa>1$ in second inequality.
Now, it is easy to see that $\tilde{p}$ and $\tilde{q}$ are globally exponentially stable, which means 
\begin{align}\label{cons-pi-qi}
 \lim\limits_{t \to \infty} p_i =& \frac{1}{N} \sum\limits_{j=1}^{N} p_j,\notag\\
  \lim\limits_{t \to \infty} q_i =& \frac{1}{N} \sum\limits_{j=1}^{N} q_j.
\end{align}
Now, using a procedure similar to proof of Theorem \ref{Theorem-NoFilter}, the variables ${S_1=\sum_{i=1}^N (p_i- r_i)}$ and ${S_2=\sum_{i=1}^N (q_i-v_i^r) }$ are defined. By summing both sides of \eqref{filter-double-i}, for $i=1,...,N$ we have 
\begin{align}\label{sum-lip}
\dot{S}_1=& S_2, \notag \\
%		\notag \\
%		q_i(t)=& \dot{z}_i(t)+v^r_i(t)
\dot{S}_2=&- \kappa S_1- \kappa S_2 \notag \\
&  -\sum\limits_{i=1}^{N} \beta_i   \mbox{sgn} \Big[ \sum\limits_{j=1}^{N} a_{ij} \Big\{ (p_i+q_i)-(p_j+q_j) \Big\} \Big].
\end{align}
Then we can use input-to-state stability to analyze the system \eqref{sum-lip} by treating the term $\sum\limits_{i=1}^{N} \beta_i   \mbox{sgn} \Big[ \sum\limits_{j=1}^{N} a_{ij} \Big\{ (p_i+q_i)-(p_j+q_j) \Big\} \Big]$ as the input and $S_1$ and $S_2$ as the states.
Since $\kappa>1$, the matrix $\begin{bmatrix}
\mathbf{0}_\mathfrak{p} & I_\mathfrak{p} \\
-\kappa I_\mathfrak{p} & -\kappa I_\mathfrak{p}
\end{bmatrix}$ is Hurwitz. 
Thus, the system \eqref{sum-lip} with zero input is exponentially stable and hence input-to-state stable and we have $S_1 \to 0$ and $S_2 \to 0$. Therefore, we have that $\lim\limits_{t \to \infty}\sum_{i=1}^N p_i = \sum_{i=1}^N r_i$ and $\lim\limits_{t \to \infty} \sum_{i=1}^
N q_i = \sum_{i=1}^N v_i^r$. Now, using \eqref{cons-pi-qi}, it is easy to see that the estimation goal \eqref{estiamtor-goal} is achieved.

\textbf{Controller}:
Here, each agent tracks its own generated signal, its own estimator output, where it is shown that by using the control inputs  \eqref{ucontrol-single-i}-\eqref{ucontrol-euler-nofliter}, we have $\lim\limits_{t \to \infty} x_i = p_i$ for $i=1,\ldots, N$.

Single-integrator:
Using the control input \eqref{ucontrol-single-i} for \eqref{single-dynamics}, we obtain the closed-loop dynamics for agents with single-integrator dynamics as
\begin{align}\label{ucontrol-close-single-i}
		\dot{\tilde{x}}_i= -  \eta_{i}& \text{sgn} (\tilde{x}_i),   \qquad  i=1,\ldots, M-1,
			\end{align}  
where $\tilde{x}_i= x_i- p_i$. Consider the candidate Lyapunov function $V_s=\frac{1}{2}\tilde{x}_i^T \tilde{x}_i$. By taking the derivative of $V_s$, we have $\dot{V}_s=-  \eta_{i} \norm{\tilde{x}_i}_1$. It is now easy to conclude that $\tilde{x}_i,$ for $i=1,\ldots, M-1,$ converges to zero. 

Double-integrator: For agents with double-integrator dynamic the closed-loop system, using the control input \eqref{ucontrol-double-i} for \eqref{double-dynamics}, can be written as
		\begin{align} \label{ucontrol-closed-double-i}
\dot{\tilde{v}}_i=&-  \eta_{i}\text{sgn} (\tilde{x}_i+\tilde{v}_i)- \eta_{i}\tilde{x}_i -\eta_{i}\tilde{v}_i, \qquad  i=M,\ldots, N'-1,
	\end{align}  
where $\tilde{v}_i = v_i - q_i$. Consider the candidate Lyapunov function
$V_d= \frac{1}{2} 
	\begin{bmatrix}
		\tilde{x}_i^T && \tilde{v}_i^T
	\end{bmatrix} 
	\begin{bmatrix}
		2\eta_i I_{\mathfrak{p}} && I_{\mathfrak{p}} \\
		I_{\mathfrak{p}} && I_{\mathfrak{p}}
	\end{bmatrix}
	\begin{bmatrix}
		\tilde{x}_i \\
		\tilde{v}_i
	\end{bmatrix}$.
Since $\eta_i>\frac{1}{2}$, $V_d$ is positive definite. The derivative of $V_d$ along system \eqref{ucontrol-closed-double-i} is obtained as
\begin{align*}
	\dot{V}_d=&  2\eta_i \tilde{x}_i^T  \tilde{v}_i+\tilde{v}_i^T \tilde{v}_i -\eta_i \tilde{x}_i^T (\tilde{x}_i+\tilde{v}_i)\\
	& -\eta_i \tilde{x}_i^T\mbox{sgn}(\tilde{x}_i+\tilde{v}_i) -\eta_i \tilde{v}^T (\tilde{x}_i+\tilde{v}_i) -\eta  \tilde{v}_i^T\mbox{sgn}(\tilde{x}_i+\tilde{v}_i) \\
	=&-\eta_i \tilde{x}_i^T  \tilde{x}_i+(1-\eta_i) \tilde{v}_i^T  \tilde{v}_i -\eta_i \|\tilde{x}_i+\tilde{v}_i\|_1.
\end{align*}
Since $\eta_i >1$, it is concluded that $\begin{bmatrix}
\tilde{x}_i \\
\tilde{v}_i
\end{bmatrix}$  for $i=M,\ldots, N'-1,$  asymptotically converges to zero. 
	
Euler-Lagrange: It follows form (P3) and \eqref{ucontrol-euler-nofliter} that the closed-loop dynamics for agent $i=N',\ldots, N,$ can be written as
\begin{align} \label{u-nofitler-closed-loop-Euler}
	M_i(x_i) \dot{s}_i+C(x_i, \dot{x}_i) s_i+&Y_i(x_i,\dot{x}_i,p_i,q_i,\dot{q}_i) \theta_i\\ &=Y_i(x_i,\dot{x}_i,p_i,q_i,\dot{q}_i)\hat{\theta}_i-\alpha s_i.\notag
	\end{align}	
Consider the candidate Lyapunov function $V_e=\frac{1}{2}s_i^T M_is_i+\frac{1}{2}\tilde{\theta}_i^T \tilde{\theta}_i$, where $\tilde{\theta}_i=\hat{\theta}_i-\theta_i$. By taking the derivative of $V_e$, we have
\begin{align*}
\dot{V}_e=&\frac{1}{2}s_i^T \dot{M}_is_i+s_i^T M_i\dot{s}_i+\tilde{\theta}_i^T \dot{\hat{\theta}}_i \notag\\
=&\frac{1}{2}s_i^T \dot{M}_is_i-s_i^T C(x_i, \dot{x}_i) s_i+s_i^T Y_i(x_i,\dot{x}_i,p_i,q_i,\dot{q}_i)\tilde{\theta}_i\\
&-\alpha s_i^T s_i -\tilde{\theta}_i^TY_i(x_i,\dot{x}_i,p_i,q_i,\dot{q}_i)^T s_i\\
=&-\alpha s_i^T s_i,
\end{align*}
where (P2) is employed to obtain the last equality. Then we can get that $s_i, \tilde{\theta}_i \in \mathbb{L}_\infty$. Also under Assumption \ref{bound-r-vr}, it is easy to see that $p_i$ and $q_i$ are bounded. Therefore, using the boundedness of $s_i$, we know $x_i$ and $\dot{x}_i$ remain bounded. Furthermore, from \eqref{filter-double-i}, we know that $\dot{q}_i$ is bounded. It follows from (P3) that
\begin{align} 
	M_i(x_i) [\mu (q_i-\dot{x}_i)+\dot{q}_i]&+C_i(x_i, \dot{x}_i) [\mu (p_i-x_i)+q_i]+g_i(x_i)\notag\\&=Y_i(x_i,\dot{x}_i,p_i,q_i,\dot{q}_i) \theta_i,
	\end{align}	
where using the 	boundedness of its components, we can see that $Y_i(x_i,\dot{x}_i,p_i,q_i,\dot{q}_i)$ is bounded for $i=N',\ldots, N$. Now, it follows form \eqref{u-nofitler-closed-loop-Euler} that $\dot{s}_i$ is bounded. This guarantees the boundedness of $\ddot{V}_e$. Thus by using Lyapunov-like lemma, we have $s_i \to 0$ for $i=N',\ldots, N$.
Using an argument similar to Lemma 5 in \cite{FeiChenACC13}, it is obtained that  $\lim\limits_{t \to \infty} x_i = p_i$ for $i=N',\ldots, N$. Till now it is proved that $\lim\limits_{t \to \infty} x_i = p_i$ for $i=1,\ldots, N$. Now, it follows from \eqref{estiamtor-goal} that the goal \eqref{goal} is achieved. 
\endproof
\begin{remark} \label{remark-euler-restrict}
Note that the restriction in Theorem \ref{thm:DAT-filter}, Assumption \ref{bound-r-vr}, is placed by agents with Euler-Lagrange dynamic. As it is stated in P3, the regression matrix $Y_i$ is a function of its own states, $x_i$ and $\dot{x}_i$. According to our goal, these states have to track, respectively, the average of the reference inputs and the reference velocities; hence to guarantee a bounded  $Y_i$, it is required to have a bounded reference input and  the reference velocity (Assumption \ref{bound-r-vr}).
\end{remark}
\begin{remark}
Both algorithms introduced in \eqref{u-single-i}-\eqref{u-euler-nofliter} and  \eqref{filter-double-i}-\eqref{ucontrol-euler-nofliter} require that Assumptions \ref{bound-r-vr} hold. In algorithm \eqref{u-single-i}-\eqref{u-euler-nofliter}, the agents just need their own positions and the relative positions between themselves and their neighbors.
In some applications, these pieces of information can be obtained by sensing; hence the communication necessity might be eliminated.
However, in algorithm \eqref{filter-double-i}-\eqref{ucontrol-euler-nofliter} each agent must communicate two variables $p_i$ and $q_i$ with its neighbors, which needs communication. 
\end{remark}
\begin{remark}
The restriction noted in Remark \ref{remark-euler-restrict} is inevitable when we have an agent with Euler-Lagrange dynamics among our agents. However, for a multi-agent system consisting of agents with only single-integrator and double-integrator dynamics, Assumption \ref{bound-r-vr} will be relaxed in algorithm  \eqref{filter-double-i}-\eqref{ucontrol-euler-nofliter}. As a result, there will be no restriction on admissible reference inputs. Note that in this framework Assumptions \ref{bound-r-vr} can not be relaxed for algorithm \eqref{u-single-i}-\eqref{u-euler-nofliter}.
\end{remark}

%{\color{red}	Define ${r(t)=[r^T_1(t),\ldots,r^T_N(t) ]^T}$, ${v^r(t)=[{v_1^r}^T(t),\ldots,{v_N^r}^T(t) ]^T}$,  and ${a(t)=[a_1^T(t),\ldots,a_N^T(t)]^T}$.
%}
\section{Simulation and Discussion } \label{sec:sim}
In this section, we present a simulation to illustrate the theoretical result in Subsection \ref{subsec:NoFilter}. Consider a team of six agents. The interaction among the agents is described by an undirected graph shown in Fig. \ref{graph}, where agents are colored based on their dynamics. Agents with single-integrator, double-integrator, and Euler-Lagrange dynamics are, respectively, colored red, blue and green. 

The agents' goal is to track the average of their reference inputs. The reference input for agent $i$ is defined as $r_i(t)=\begin{bmatrix}
3 i \text{sin}(\frac{\pi}{25}t)\\4 i \text{cos}(\frac{\pi}{50}t)
\end{bmatrix}$. The reference input and its velocity is bounded and Assumption \ref{bound-r-vr} is satisfied.
The dynamic for agents with single-integrator and double-integrator dynamics is defined as \eqref{single-dynamics} and \eqref{double-dynamics}. The dynamic equation for each Euler-Lagrange agent is modeled by $m_i\ddot{x}_i+c_i\dot{x}_i=u_i, i=5, 6$, where $x_i(t)$ is the coordinate of agent $i$ in $2D$ plane \cite{Cheah}. The parameters $m_i$ and $c_i$ represent, respectively, the mass
and damping constants of the agent $i$, which are assumed to be
constant but unknown. We let $m_1=1, c_1=0.5, m_2=1.5$, and $c_2=0.6$. 

In our example, we apply the algorithm \eqref{u-single-i}-\eqref{u-euler-nofliter}, where the controllers' parameters are selected as  $\beta_i=25, \forall i \in \mathcal{I}$, and $\alpha=15$. The initial positions of the agents are selected as $[8\ 0]^T,  [9 \ 3]^T, [10 \ 6]^T, [11 \ 9]^T,  [12 \ 12]^T,$ and $[13 \ 15]^T$ and their initial velocities are selected as zero. 
Fig. \ref{figsim} shows that the distributed
average tracking is achieved and agents track the average of the reference inputs.
\begin{figure}[t]
\begin{center} \hspace*{-.4cm}
\vspace{0.1cm} {\scalebox{0.2800}{\includegraphics*{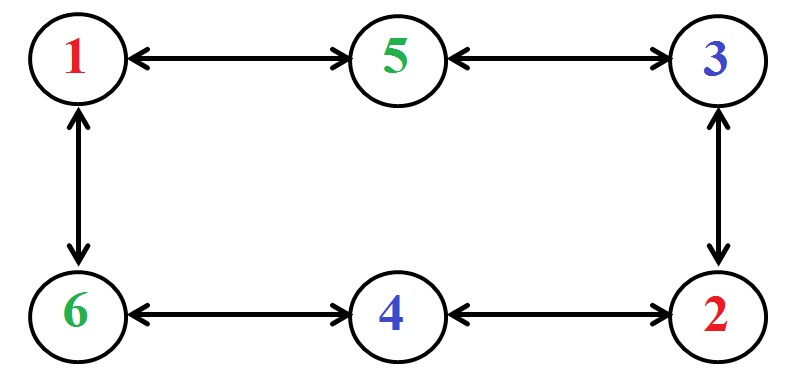}}}
\vspace{-0.3cm} \hspace{-0.4cm}  \caption{Undirected graph} \vspace{-0.1cm}
\label{graph}
\end{center}
\end{figure}
\begin{figure}[t]
\begin{center} \hspace*{-0.2cm}
{\scalebox{0.27}{\includegraphics*{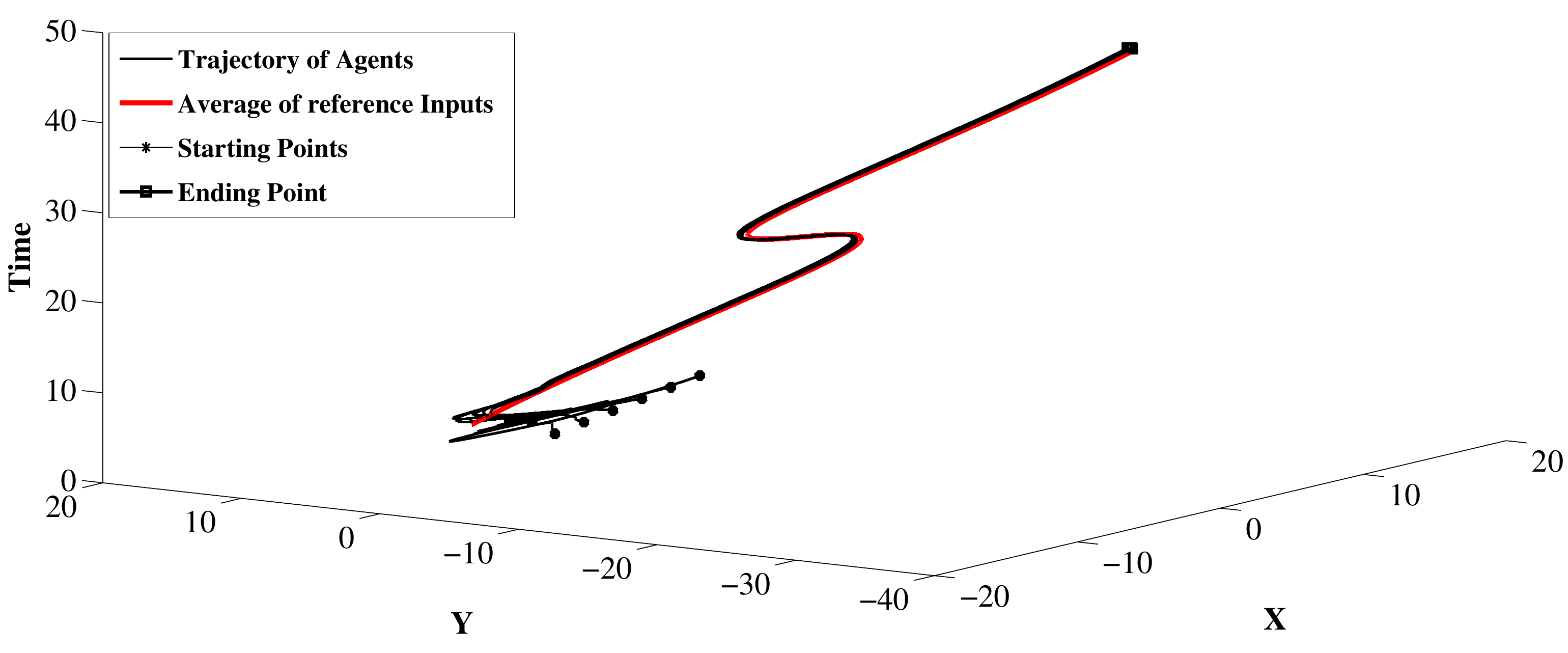}}}
\vspace{-.1cm} \caption{Trajectories of all agents along with the average of the reference inputs using the algorithm \eqref{u-single-i}-\eqref{u-euler-nofliter}.} \vspace{-0.1cm}
\label{figsim}
\end{center}
\end{figure}

%{\color{blue}$\dot{\tilde{\nu }}_i=\upsilon_i$}
\section{Conclusions} \label{sec:conclusions}

In this paper, a distributed average tracking was studied for a group of heterogeneous physical agents. The multi-agent system was consisted of the agents with single-integrator, double-integrator and Euler-Lagrange dynamics.
Two nonsmooth algorithms were proposed to achieve the distributed average tracking goal. In our first proposed algorithm, each agent required to have access to only its own position and the relative positions between itself and its neighbors, where it was possible
to rely on only local sensing. To relax some restrictive assumptions on admissible reference inputs, we proposed the second algorithm, where a filter was introduced for each agent to generate an estimation of the average reference inputs. Then, each agent tracked its own generated signal.

\bibliographystyle{IEEEtran}
\bibliography{IEEEabrv,refs}
%%%%%%%%%%%%%%%% end Bibliography %%%%%%%%%%%%%%%%%

\end{document}